\documentclass[11pt]{article}
\usepackage{amssymb,latexsym,amsmath,amsbsy,amsthm,amsxtra,amsgen}
\newfont{\msbm}{msbm10 at 11pt}

\newcommand {\Z} {\mbox{\msbm Z}}

\def\be{\begin{equation}}
\def\ee{\end{equation}}
\def\ba{\begin{align}}
\def\ea{\end{align}}

\newtheorem{Theo}{Theorem}
\newtheorem{Lemma}[Theo]{Lemma}

\newtheorem{Prop}[Theo]{Proposition}

\begin{document}

\title{Lower Bound on the Rate of Adaptation in an Asexual Population}
\author{Michael Kelly\footnote{Simpson University, 2211 College View Drive, Redding, CA 96003, E-mail address: mkelly@simpsonu.edu \newline AMS 2010 Subject Classifications. Primary 92D15; Secondary 60J27, 60K35, 92D10. \newline Keywords: Evolutionary process, Moran model, selection, adaptation rate.}}

\maketitle

\begin{abstract}
We consider a model of asexually reproducing individuals with random mutations and selection. The rate of mutations is proportional to the population size, $N$. The mutations may be either beneficial or deleterious. In a paper by Yu, Etheridge and Cuthbertson (2009) it was conjectured that the average rate at which the mean fitness increases in this model is $O(\log N/(\log\log N)^2)$.  In this paper we show that for any time $t > 0$ there exist values $\epsilon_N \rightarrow 0$ and a fixed $c > 0$ such that  the maximum fitness of the population is greater than $cs\log N/(\log\log N)^2$ for all times $s \in [\epsilon_N,t]$ with probability tending to 1 as $N$ tends to infinity.
\end{abstract}

\section{Introduction}
In this paper we consider an evolutionary model of a population of individuals that was first introduced by Yu, Etheridge and Cuthbertson \cite{YEC}. The model has parameters $N,q,\mu$ and $\gamma$. The population size is a fixed integer $N > 0$. The rate at which each individual gets mutations is $\mu > 0$ and each mutation is beneficial with probability $q$, where $0 < q \leq 1$, and deleterious with probability $1-q$. Each beneficial mutation increases an individual's selective advantage by $\gamma > 0$.

To make the description precise, the stochastic process we are interested in is
$$(X_t : t \geq 0) \mbox{ where } X_t = (X_t^1,X_t^2, \dots, X_t^N) \in \Z^N$$
is an $N$-dimensional vector for $t \geq 0$. Each coordinate represents the fitness of an individual in the population at time $t$. For all coordinates $i,j \in \{1,2,\dots,N\}$ the model has the following transition rates:
\begin{align*}
& X^i \rightarrow X^i+1  \mbox{ at rate } q\mu, \\
& X^i \rightarrow X^i-1 \mbox{ at rate } (1-q)\mu, \\
& X^i \rightarrow X^j \mbox { at rate } \frac{1}{N}, \\
& X^i \rightarrow X^j \mbox { at rate } \frac{\gamma}{N}(X^j-X^i) \vee 0.
\end{align*}
The first two transition rates correspond to beneficial and deleterious mutations respectively. The rate $1/N$ at which $X^i \rightarrow X^j$ is the resampling mechanism. At rate one an individual is chosen to give birth uniformly at random and an individual is chosen to die uniformly at random. The rate $\gamma(X^j - X^i)/N \vee 0$ at which $X^i \rightarrow X^j$ is the selection mechanism. Mutation, resampling and selection events occur independently of one another.

With the vector valued process $X$ in mind we can think of $X^i$ as the fitness of individual $i$. That is, if individual 1 dies and individual 2 gives birth, rather than thinking of the population as having lost one individual and gaining a new individual we think of individual 1 as adopting the fitness of individual 2. From the perspective of $X$ as a vector valued process, to say individual $i$ has died and individual $j$ has given birth we only mean that the value of $X^i$ has been changed to the value of $X^j$. This way, we can continue referring to the individual whose fitness is recorded in coordinate $i$ as the $i$th individual even though from the biological perspective the individual whose fitness is in the $i$th coordinate will change over time due to death and birth events.

It is important to note that with the vector valued perspective in mind, when we say individual $i$ dies due to resampling we mean that the value in coordinate $i$ has changed its value to a randomly chosen coordinate. By the dynamics of the model, such events occur in each coordinate at rate 1. Each individual can only decrease in fitness when it gets a deleterious mutation or dies due to a resampling event (although death due to a resampling event does not necessarily mean that the fitness of the individual will decrease). The rate at which events occur that may cause an individual's fitness to decrease is
$$d = (1+q)\mu + 1.$$
The value $d$ is the same for each individual and does not depend on the population size $N$.

\subsection{Previous Results}

Let
$$\overline{X}_t = \frac{1}{N} \sum_{i=1}^N X_t^i$$
be the mean fitness of the process at time $t$. Let $(X_t^C : t \geq 0)$ be the stochastic process centered about its mean so that the $i$th coordinate of $X_t^C$ is
$$X_t^{i,C} = X_t^i - \overline{X}_t.$$
It was shown in \cite{YEC} that the centered process is ergodic. Let $\pi$ be the stationary distribution. If we choose $X_0^C$ according to the stationary distribution then the expected variance of the fitnesses,
$$E^\pi [c_2] \mbox{ where } c_2 = \frac{1}{N}\sum_{i=1}^N (X_t^{i,C})^2,$$
does not change in time. Equation (12) of \cite{YEC} states that
$$E^\pi[\overline{X}_t] = (\mu(2q-1)+sE^\pi[c_2])t$$
where $E^\pi$ means that the initial value of the process was chosen according to its stationary distribution. This establishes a constant rate of increase of $E^\pi[\overline{X}_t]$ in time. Computing $E^\pi[c_2]$ would give the rate of adaptation. This sort of technique was used by Desai and Fisher \cite{DF} for a related model. However, computing $E^\pi[c_2]$ is a difficult problem so we take another approach.

A heuristic argument in \cite{YEC} showed the rate of adaptation should be $O(\log N/(\log\log N)^2)$. They were able to establish a rigorous lower bound which showed $E[\overline{X}_t]$ increases at a rate of $O(\log^{1-\delta} N)$ as $t$ tends to infinity for large enough values of $N$. In Kelly \cite{ROAUp} it was shown that
$$\frac{E[\overline{X}_t]}{t} \leq \frac{C\log N}{(\log\log N)^2}$$
for some constant $C > 0$.

Similar results have been conjectured for the rate of adaptation in the biological literature. For examples, see Rouzine, Brunet and Coffin \cite{RBC}, Brunet, Rouzine and Wilke \cite{BRW} and Rouzine, Brunet and Wilke \cite{RBW}. Technical, but not rigorous, arguments are given for the mean rate of adaptation of similar models. The arguments rely on two assumptions. First, one assumes that the bulk of the population behaves deterministically and the random noise is concentrated only in the highest fitness classes. Second, one assumes that the bulk of the population is shaped like a Gaussian curve. The one or two highest fitness classes behave as stochastic populations that are expected grow due to the selective advantage they have over the bulk of the population. The rate of adaptation can be estimated by computing the mean time to the creation of a new highest fitness class.

\subsection{Main Result}

Fix a time $t > 0$. Let
$$\mathcal{T} = \frac{16(\log\log N)^2}{\gamma \log N}$$
represent a time value and let $M$ be the integer such that
$$M < \frac{t}{2\mathcal{T}} \leq M+1.$$
Define
$$X_s^+ = \max\{X_s^i : 1 \leq i \leq N\}$$
for all $t \geq 0$.

\begin{Theo} \label{mainTheo}
Let
$$\mathcal{E} = \left\{X_s^+ \geq \frac{Ms}{2} \mbox{ for } s \in [\mathcal{T},t]\right\}.$$
Then
$$\lim_{N \rightarrow \infty} P(\mathcal{E}) = 1.$$
\end{Theo}
Establishing Theorem \ref{mainTheo} shows that with probability tending to 1 the maximum fitness of the population is at least $O(\log N/(\log\log N)^2)$. The result differs from the results in \cite{YEC} and \cite{ROAUp} in that it is a convergence in probability rather than a convergence of the mean fitness. Also, the result agrees with the heuristic argument given in \cite{YEC}.

\section{Proof of the Theorem}

For the remainder of the paper $t > 0$ is a fixed time value. Define the sequence $(s_i : i \geq 0)$ by
$$s_i = 2i\mathcal{T}.$$
Let $M$ be the integer such that
$$s_M < t \leq s_{M+1}.$$
Note that $M$ tends to infinity as $N$ does because $t$ is fixed and $\mathcal{T}$ tends to 0 as $N$ tends to infinity.

At the times $s_i$ we wish to label an individual with maximum fitness. There may be multiple individuals that have the maximum fitness. To make a precise choice for an individual with maximum fitness but arbitrary with respect to which index, for $i \geq 0$ define
$$\alpha(i) = \min\{j : X_{s_i}^j = X_{s_i}^+\}.$$
When $j = \alpha(i)$ we say individual $j$ is labelled $\alpha$ over the timer interval $[s_i, s_{i+1})$. Likewise, we want to label an individual that is considered to be the second most fit individual in the population. The second most fit individual may have the same fitness as the most fit individual. For $i \geq 0$ define
$$\beta(i) = \min\{j \neq \alpha(i): X_{s_i}^j \geq X_{s_i}^k \mbox{ for } k \in \{1,2,\dots,N\}\backslash\{\alpha(i)\}\}.$$

We define a Poisson point process $\mathcal{P}_\alpha$ that marks times at which individuals labelled $\alpha$ may decrease in fitness. If individual $\alpha(i)$ gets a deleterious mutation or dies due to a resampling event at time $t \in [s_i, s_{i+1})$ there is a point of $\mathcal{P}_\alpha$ at time $t$. Similarly, we define a Poisson point process $\mathcal{P}_\beta$ for the individuals $\beta(i)$, $i \geq 0$. Both $\mathcal{P}_\alpha$ and $\mathcal{P}_\beta$ are Poisson point processes on $[0,\infty)$ of intensity $d$. The point processes take values on the measure space $(G,\mathcal{G})$ where
$$G = \{x \subset [0,\infty) : |x \cap B| < \infty \mbox{ for all } B \in \mathcal{B} \mbox{ such that } \lambda(B) < \infty\}$$
where $\mathcal{B}$ is the Borel $\sigma$-field on $[0,\infty)$ and $\lambda$ is the Lebesgue measure and
$$\mathcal{G} = \sigma(\{x \in G : |x \cap B| = m\} : B \in \mathcal{B}, m \geq 0).$$
Let $\mathcal{P} = \mathcal{P}_\alpha \cup \mathcal{P}_\beta$. Since $\mathcal{P}_\alpha$ and $\mathcal{P}_\beta$ are independent, $\mathcal{P}$ is a Poisson point process of rate $2d$ that marks times at which individuals labelled $\alpha$ or $\beta$ die. For $K \geq 0$ define
$$\Lambda_K = \{x \in G : |x \cap [0,s_{M+1}]| \leq K, |x \cap [0,s_1]| = 0 \mbox{ and } |x \cap [s_i,s_{i+1})| \leq 1 \mbox{ for } 1 \leq i \leq M\}.$$
Note that $\Lambda_K \in \mathcal{G}$ for $K \geq 0$.

\begin{Prop} \label{PLk}
Let $\epsilon > 0$. There exists $K \geq 0$ such that $P(\mathcal{P} \in \Lambda_K) > 1-\epsilon$ for large enough values of $N$.
\end{Prop}

\begin{proof}
Let $\delta > 0$ such that $(1-\delta)^2 > 1-\epsilon$. For an interval $I$ we define $\mathcal{P}I$ to be the number of points of $\mathcal{P}$ in the interval $I$. Then $\mathcal{P}[0,s_{M+1}]$ has the Poisson distribution with mean $2ds_{M+1}$. Since
$$\lim_{N \rightarrow \infty} s_{M+1} = t,$$
for $N$ large enough $s_{M+1} < t+1$. Because $t+1$ is fixed, there exists $K$ large enough so that
$$\sum_{k=0}^KP(\mathcal{P}[0,t+1] = k) > 1-\delta.$$

We expand over the conditional probability to get
$$P(\mathcal{P} \in \Lambda_K) = \sum_{k=0}^K P(\mathcal{P} \in \Lambda_K|\mathcal{P}[0,t+1] = k)P(\mathcal{P}[0,t+1] = k).$$
On the event $\{\mathcal{P}[0,t+1] = k\}$ the $k$ points $U_1, \dots, U_k \in \mathcal{P}$ are i.i.d. and have the uniform distribution on $[0,t+1]$. Therefore,
$$P(\mathcal{P} \in \Lambda_K|\mathcal{P}[0,t+1] = k) \geq \prod_{i=1}^{k} \frac{t+1-2i\mathcal{T}}{t+1}.$$
Because $\mathcal{T}$ tends to 0 as $N$ tends to infinity, $P(\mathcal{P} \in \Lambda_K|\mathcal{P}[0,t+1] = k)$ tends to 1 for any fixed $k$. Since $K$ is fixed, for large enough values of $N$ we have
$$P(\mathcal{P} \in \Lambda_K|\mathcal{P}[0,t+1] = k) > 1-\delta \mbox{ for } 0 \leq k \leq K.$$
By our choice of $K$, for $N$ large enough
\begin{align*}
P(\mathcal{P} \in \Lambda_K) & = \sum_{k=0}^K P(\mathcal{P} \in \Lambda_K|\mathcal{P}[0,t+1] = k)P(\mathcal{P}[0,t+1] = k) \\
& > \sum_{k=0}^K (1-\delta)P(\mathcal{P}[0,t+1] = k) \\
& > (1-\delta)^2 \\
& > 1-\epsilon.
\end{align*}
\end{proof}

For $i \geq 0$ we define stopping times $\tau_\alpha(i)$, $\tau_\beta(i)$ and $\tau(i)$ so that we can couple the progeny of individual $\alpha(i)$ with a branching process over the time interval $[s_i, (s_i+\mathcal{T}) \wedge \tau_\alpha(i) \wedge \tau(i))$ and $\beta(i)$ with a branching process over the time interval $[s_i, (s_i+\mathcal{T}) \wedge \tau_\beta(i) \wedge \tau(i))$. Define
$$\tau_\alpha(i) = \inf\{s \geq s_i : s \in \mathcal{P}_\alpha\} \mbox{ and } \tau_\beta(i) = \inf\{s \geq s_i : s \in \mathcal{P}_\beta\}.$$
Define a process $(W_s : s \geq 0)$ by $W_s = k$ if at least $N/2$ individuals have fitness in $[k,\infty)$ at time $s$ but fewer than $N/2$ individuals have fitness in $(k,\infty)$ at time $s$. Let
$$\mathcal{W} = \frac{\log N}{8\log\log N}$$
and
$$\tau(i) = \inf\left\{s \geq s_i : X_{s_i}^{\beta(i)} - W_s < \mathcal{W}\right\}.$$

Let $(\mathcal{Z}_s : s \geq 0)$ be a multi-type branching process that initially has one particle of type 0. Particles have death rate $d$ and give birth to one offspring at rate 
$$w = \frac{\gamma \mathcal{W}}{2}.$$
Particles change from type-0 to type-1 at rate $q\mu$. Let $(\mathcal{Z}_s^\alpha(i), \mathcal{Z}_s^\beta(i) : i \geq 0)$ be a collection of i.i.d. branching processes each having the same distribution as $\mathcal{Z}$. We couple $X$ with the branching processes $\mathcal{Z}_s^\alpha(i)$ over the time intervals $[s_i, (s_i+\mathcal{T}) \wedge \tau_\alpha(i) \wedge \tau(i))$ as follows:\\

\begin{itemize}
\item At time $s_i$ the particle in $\mathcal{Z}_0^\alpha(i)$ is coupled with individual $\alpha(i)$. Over the time interval $[0,((s_i+\mathcal{T}) \wedge \tau_\alpha(i) \wedge \tau(i))-s_i)$ each particle of $\mathcal{Z}^\alpha(i)$ will be coupled with an individual in $X$.
\item Particles in $\mathcal{Z}$ increase in type by 1 at time $s$ if and only if the corresponding individual in $X$ gets a beneficial mutation at time $s_i+s$.
\item The one particle in $\mathcal{Z}_0^\alpha(i)$ dies at rate $d$ independent of the process $X$. Any other particle of $\mathcal{Z}^\alpha(i)$ dies at time $s$ if and only if the individual it is coupled with gets a deleterious mutation or dies due to a resampling event at time $s_i+s$.
\item To explain the branching events of $\mathcal{Z}^\alpha(i)$ we give another description of the selection mechanism. For each individual $i$ in $X$ define a Poisson process $\mathcal{P}^i$ on $[0,\infty) \times [0,\infty)$. The collection of Poisson processes $(\mathcal{P}^i : 1 \leq i \leq N)$ are independent of one another and each have intensity $\lambda \times \frac{\gamma}{N} \lambda$ where $\lambda$ is Lebesgue measure. For a fixed individual $j$, define
$$\Phi_s^j = \sum_{k=1}^N (X_s^j - X_s^k)^+.$$
Individual $j$ gives birth at time $s$ if there is a point of $\mathcal{P}^j$ at $(s,x)$ for some $x \leq \Phi_s^j$. On the event that individual $j$ gives birth individual $k$ is chosen to die with probability
$$\frac{(X_s^j - X_s^k)^+}{\Phi_s^j}.$$

Define $\Theta_s^j$ to be the sum of terms in $\Phi_s^j$ which are at least as large as $\mathcal{W}$:
$$\Theta_s^j = \sum_{k=1}^N (X_s^j - X_s^k) 1_{\{X_s^j - X_s^k \geq \mathcal{W}\}}.$$
For $s_i \leq s < \tau(i)$ we have $\{W_t \geq \mathcal{W}\}$ so
$$\Theta_s^j \geq \frac{N\mathcal{W}}{2}$$
for each individual $j$ such that $X_s^j \geq X_{s_i}^{\beta(i)}$.

If a particle in $\mathcal{Z}$ is coupled with individual $j$ then the fitness of individual $j$ is at least $X_{s_i}^\beta(i)$.  The particle will branch at time $s$ with probability
$$\frac{N\mathcal{W}}{2\Theta_s^j}$$
if individual $j$ gives birth at time $s_i+s$ and the individual $k$ chosen to die satisfies $X_s^j - X_s^k \geq \mathcal{W}$. The new particle is coupled with individual $k$. \end{itemize}
Similarly, the progeny of individual $\beta(i)$ is coupled with the branching process $\mathcal{Z}_s^\beta(i)$ over the time intervals $[s_i, (s_i+\mathcal{T}) \wedge \tau_\beta(i) \wedge \tau(i))$.

Let $(Z_s : s \geq 0)$ be the number of particles in the process $\mathcal{Z}$ of any type at time $s$. The following lemmas establish some results for the distribution of $\mathcal{Z}_s$.

\begin{Prop} \label{branchProp}
As $N$ tends to infinity
$$wP(Z_{\mathcal{T}} = 0) \rightarrow d \mbox{ and } (\log N) P(Z_{\mathcal{T}} = 1) \rightarrow 1.$$
\end{Prop}

\begin{proof}
Let
\begin{align}
& f(s) = \frac{d(e^{(w-d)s}-1)}{w e^{(w-d)s}-d} \hspace{0.2in}\mbox{ and}\label{ffun} \\
& g(s) = \frac{w(e^{(w-d)s}-1)}{w e^{(w-d)s}-d} = 1 - \frac{w - d}{w e^{(w-d)s}-d}. \label{gfun}
\end{align}

We use the generating function of birth-death processes from page 109 of Athreya and Ney \cite{AN} with birth rate $w$ and death rate $d$ to get
$$F(x,s) = \frac{d(x-1)-(w x-d)e^{(d-w)s}}{w(x-1)-(w x-d)e^{(d-w)s}}$$
where
$$F(x,s) = \sum_{k=0}^\infty P(Z_s = k | Z_0 = 1)x^k.$$
Solving the equation
$$\frac{1}{k!} \left[\frac{\delta^k}{\delta x^k} F(x,s) \right|_{x = 0} = P(Z_s = k | Z_0 = 1)$$
we arrive at the formulas
\begin{align}
& P(Z_s = 0| Z_0 = 1) = f(s) \hspace{0.2in}\mbox{ and} \nonumber \\
&  P(Z_s = i|Z_0 = 1) = (1-f(s))(1-g(s))g(s)^{i-1} \mbox{ for } i \geq 1.\label{popEq}
\end{align}

The probability that the process dies by time $\mathcal{T}$ is then
$$f(\mathcal{T}) = \frac{d(e^{(w-d)\mathcal{T}}-1)}{w e^{(w-d)\mathcal{T}}-d}.$$
Because $w\mathcal{T} = \log\log N$ tends to infinity with $N$ we have
$$\lim_{N \rightarrow \infty} wP(Z_{\mathcal{T}} = 0) = \lim_{N \rightarrow \infty} wf(\mathcal{T}) = d.$$

The probability that there is exactly one individual at time $\mathcal{T}$ is 
$$(1-f(\mathcal{T}))(1-g(\mathcal{T})) = (1-f(\mathcal{T})) \cdot \frac{w-d}{w e^{(w-d)\mathcal{T}}-d}.$$ 
By definition,
$$\frac{1}{e^{w\mathcal{T}}} = \frac{1}{\log N}.$$
Therefore,
$$\lim_{N \rightarrow \infty} (\log N)P(Z_\mathcal{T} = 1) = \lim_{N \rightarrow \infty} (\log N)(1-f(\mathcal{T}))(1-g(\mathcal{T})) = 1.$$
\end{proof}

\begin{Prop} \label{AkProp}
As $N$ tends to infinity,
$$P(Z_{\mathcal{T}} > \mathcal{W}) \rightarrow 1.$$
\end{Prop}

\begin{proof}
By Equation (\ref{popEq}) we have
\begin{align*}
P(Z_s \leq \mathcal{W} | Z_0 = 1) & = f(s) + (1-f(s))(1-g(s))\sum_{i=1}^\mathcal{W} g(s)^{i-1} \\
& = f(s)+(1-f(s))(1-g(s)^\mathcal{W}).
\end{align*}
Therefore,
\begin{align*}
P(Z_\mathcal{T} > \mathcal{W} | Z_0 = 1) & = (1-f(\mathcal{T}))g(\mathcal{T})^\mathcal{W} \\
& = (1-f(\mathcal{T}))\left(1-\frac{w-d}{we^{(w-d)\mathcal{T}}-d}\right)^\mathcal{W}. \\
\end{align*}
As $N$ tends to infinity $1-f(\mathcal{T})$ tends to 1. For $g(\mathcal{T})^\mathcal{W}$ to tend to 1 it is enough to have
$$\frac{(w-d)\mathcal{W}}{we^{(w-d)\mathcal{T}}-d} \rightarrow 0.$$
Because
$$\lim_{N \rightarrow \infty} \frac{\mathcal{W}}{e^{w\mathcal{T}}} = \lim_{N \rightarrow \infty} \frac{1}{8\log \log N} = 0$$
we have
$$\lim_{N \rightarrow \infty} \frac{(w-d)\mathcal{W}}{we^{(w-d)\mathcal{T}}-d} = 0.$$
\end{proof}

\begin{Lemma} \label{smallWLem}
Let $\tau$ be a stopping time and $\Gamma \subset \{1,2,\dots,N\}$ such that $|\Gamma| \geq N/2$. Let $A_i$ be the event that individual $i$ gets at least $2\mathcal{W}$ beneficial mutation events over the time interval $[\tau,\tau+\mathcal{T}]$. Then
$$\lim_{N \rightarrow \infty}P\left(\sum_{i \in \Gamma} 1_{A_i} \geq 2\right) = 1.$$
\end{Lemma}

\begin{proof}
First divide $\Gamma$ into two disjoint sets, $\Gamma_1$ and $\Gamma_2$, each of size at least $N/4$. Because $X$ is a strong Markov process and $\tau$ is a stopping time the mutation events of $X$ that occur in the time interval $[\tau, \tau+\mathcal{T}]$ are independent of the events which occur outside of that interval and have the same distribution of those that occur over the interval $[0,\mathcal{T}]$.

Using the lower bound
\begin{align*}
P\left(\sum_{i \in \Gamma} 1_{A_i} \geq 2\right) & = P\left(\sum_{i \in \Gamma_1} 1_{A_i} + \sum_{i \in \Gamma_2} 1_{A_i}\geq 2\right) \\
& \geq P\left(\left(\sum_{i \in \Gamma_1} 1_{A_i} \geq 1\right) \cap \left(\sum_{i \in \Gamma_2} 1_{A_i}\geq 1\right)\right) \\
\end{align*}
it is enough to show
$$\lim_{N \rightarrow \infty} P\left(\sum_{i \in \Gamma_1} 1_{A_i} \geq 1\right) = 1.$$

Beneficial mutations occur on an individual at rate $q\mu$ so the probability that an individual gets $k$ or more beneficial mutations in the time interval $[\tau, \tau+\mathcal{T}]$ is
$$\sum_{j=k}^\infty \frac{(q\mu\mathcal{T})^j e^{-q\mu\mathcal{T}}}{j!}.$$

Because the beneficial mutation events of one individual occur independently of other individuals,
\begin{align*}
P\left(\bigcup_{i \in \Gamma_1} A_i\right) & = 1-\left(1- \sum_{j=\lceil2\mathcal{W}\rceil}^\infty \frac{(q\mu\mathcal{T})^j e^{-q\mu\mathcal{T}}}{j!}\right)^{N/4} \\
& \geq 1-\left(1-\frac{(q\mu\mathcal{T})^{\lceil2\mathcal{W}\rceil} e^{-q\mu\mathcal{T}}}{\lceil2\mathcal{W}\rceil!}\right)^{N/4}.
\end{align*}
It is enough to have
$$\lim_{N \rightarrow \infty}\frac{N(q\mu\mathcal{T})^{\lceil2\mathcal{W}\rceil} e^{-q\mu\mathcal{T}}}{4\lceil2\mathcal{W}\rceil!} = \infty.$$
Applying Stirling's formula and the identity 
$$\mathcal{T} = \frac{2\log\log N}{\gamma\mathcal{W}}$$
it is enough to show
$$\frac{N(2q\mu\log\log N)^{\lceil2\mathcal{W}\rceil} e^{-q\mu\mathcal{T}-\lceil2\mathcal{W}\rceil}}{4(\gamma\mathcal{W})^{\lceil2\mathcal{W}\rceil}(2\pi\lceil2\mathcal{W}\rceil)^{1/2}\lceil2\mathcal{W}\rceil^{\lceil2\mathcal{W}\rceil}}$$
tends to infinity as $N$ does. By dropping the ceiling function in the denominator we get the lower bound
$$\frac{N(2q\mu\log\log N)^{\lceil2\mathcal{W}\rceil} e^{-q\mu\mathcal{T}-\lceil2\mathcal{W}\rceil}}{4(\gamma\mathcal{W})^{\lceil2\mathcal{W}\rceil}(2\pi\lceil2\mathcal{W}\rceil)^{1/2}\lceil2\mathcal{W}\rceil^{\lceil2\mathcal{W}\rceil}} > \frac{N(2q\mu\log\log N)^{\lceil2\mathcal{W}\rceil} e^{-q\mu\mathcal{T}-\lceil2\mathcal{W}\rceil}}{4(2\gamma)^{2\mathcal{W}}(4\pi\mathcal{W})^{1/2}\mathcal{W}^{4\mathcal{W}}}.$$
By definition of $\mathcal{W}$ we have
$$\mathcal{W}^{4\mathcal{W}} = N^{1/2 - \log(8\log\log(N))/2\log\log(N)}$$
so
$$\frac{N(2q\mu\log\log N)^{\lceil2\mathcal{W}\rceil} e^{-q\mu\mathcal{T}-\lceil2\mathcal{W}\rceil}}{4(2\gamma)^{2\mathcal{W}}(4\pi\mathcal{W})^{1/2}\mathcal{W}^{4\mathcal{W}}} >  \frac{N^{1/2 + \log(8\log\log(N))/2\log\log(N)}(2q\mu\log\log N)^{\lceil2\mathcal{W}\rceil} e^{-q\mu\mathcal{T}-\lceil2\mathcal{W}\rceil}}{4(2\gamma)^{2\mathcal{W}}(4\pi\mathcal{W})^{1/2}}$$
which tends to infinity as $N$ tends to infinity.
\end{proof}

\begin{Lemma} \label{brLem}
Let $\Gamma \subset \{1,2,\dots,N\}$ such that $|\Gamma| \geq \mathcal{W}$. Let $s$ be a fixed time and let
$$A_i = \{X_{s+\mathcal{T}}^i - X_s^i \geq 1\}.$$
Then
$$\lim_{N \rightarrow \infty} P\left(\sum_{i \in \Gamma} 1_{A_i} \geq 2\right) = 1.$$
\end{Lemma}
\begin{proof}
Because $|\Gamma| > \mathcal{W}$ we can partition $\Gamma$ as $\Gamma = \Gamma_1 \cup \Gamma_2$ where $|\Gamma_1| > \mathcal{W}/2$ and $|\Gamma_2| > \mathcal{W}/2$. Using the lower bound
\begin{align*}
P\left(\sum_{i \in \Gamma} 1_{A_i} \geq 2\right) & = P\left(\sum_{i \in \Gamma_1} 1_{A_i} + \sum_{i \in \Gamma_2} 1_{A_i}\geq 2\right) \\
& \geq P\left(\left(\sum_{i \in \Gamma_1} 1_{A_i} \geq 1\right) \cap \left(\sum_{i \in \Gamma_2} 1_{A_i}\geq 1\right)\right) \\
\end{align*}
it is enough to show
$$\lim_{N \rightarrow \infty} P\left(\sum_{i \in \Gamma_1} 1_{A_i} \geq 1\right) = 1.$$

Because $X$ is a Markov process the rates of birth, death and mutation are the same as if the process were started at time 0. Individuals get beneficial mutations at rate $q\mu$. Events on which individuals may decrease in fitness occur at rate $d$. Because the event $A_i$ occurs if individual $i$ gets at least one beneficial mutation and has no deleterious mutations or deaths due to resampling events,
$$P(A_i) \geq (1-e^{-q\mu\mathcal{T}})e^{-d\mathcal{T}}$$
for all $i$. Mutation events and death due to resampling occurs on individuals independently. Therefore,
$$P\left(\sum_{i \in \Gamma_1} 1_{A_i} \geq 1\right) \geq 1-(1-(1-e^{-q\mu\mathcal{T}})e^{-d\mathcal{T}})^{\mathcal{W}/2}.$$

To finish the proof it is enough to have
$$\lim_{N \rightarrow \infty} \frac{\mathcal{W}}{2}(1-e^{-q\mu\mathcal{T}}) = \infty.$$
Using the bound $1-e^{-x} > x-x^2/2$ we have
$$\frac{\mathcal{W}}{2}(1-e^{-q\mu\mathcal{T}}) > \frac{2\mathcal{W}q\mu\mathcal{T} - \mathcal{W}(q\mu\mathcal{T})^2}{2}.$$
Since
$$\mathcal{W}\mathcal{T} = \frac{2\log\log N}{\gamma}$$
we have
$$\lim_{N \rightarrow \infty} \frac{\mathcal{W}}{2}(1-e^{-q\mu\mathcal{T}}) = \infty.$$
\end{proof}

\begin{Prop} \label{PEx}
Let $\epsilon > 0$. For any $K \geq 0$ there exists $N_0$ such that for $N \geq N_0$ we have
$$P(\mathcal{E}|\mathcal{P} = x) > 1-\epsilon$$
for all $x \in \Lambda_K$.
\end{Prop}
\begin{proof}
Fix an integer $0 \leq i \leq M$. Let $\sigma_\alpha(i) = (s_i+\mathcal{T}) \wedge \tau_\alpha(i) \wedge \tau(i)$ and let $\sigma_\beta(i) = (s_i+\mathcal{T}) \wedge \tau_\alpha(i) \wedge \tau(i)$. By definition of $\Lambda_K$, if $x \in \Lambda_K$ then either $\sigma_\alpha(i) \neq \tau_\alpha(i)$ or $\sigma_\beta(i) \neq \tau_\beta(i)$. Therefore, at least one of the following must occur:
\begin{enumerate}
\item $\tau(i) < s_i+\mathcal{T}$,
\item $\sigma_\alpha(i) = s_i+\mathcal{T}$ or
\item $\sigma_\beta(i) = s_i+\mathcal{T}$.
\end{enumerate}

First suppose $\tau(i) < s_i + \mathcal{T}$. Then, for $x \in \Lambda_K$, at least one of $\sigma_\alpha(i) = \tau(i)$ or $\sigma_\beta(i) = \tau(i)$. Without loss of generality suppose $\sigma_\alpha(i) = \tau(i)$. At time $\tau(i)$ there are at least $N/2$ individuals with fitness $X_{s_i}^{\beta(i)}-\mathcal{W}$ or higher. Let $\Gamma$ be the set of all individuals that at time $\tau(i)$ have fitness $X_{s_i}^{\beta(i)}-\mathcal{W}$ or higher. By Lemma \ref{smallWLem}, with probability tending to 1 there will be two individuals of $\Gamma$ that receive $2\mathcal{W}$ beneficial mutations by time $\tau_i+\mathcal{T}$.  Both of these individuals have fitness greater than $X_{s_i}^\beta(i)$ at time $\tau_i+\mathcal{T}$. The probability that these individuals get a deleterious mutation event or decrease due to a resampling event over the time interval $[s_i,s_{i+1}]$ is $(1-e^{-2d\mathcal{T}})$ which tends to 0 as $N$ tends to infinity. Therefore, with probability $p_1$ which tends to 1 as $N$ tends to infinity we have $X_{s_{i+1}}^{\beta(i+1)} > X_{s_i}^{\beta(i)}$.

If $\tau(i) > s_i + \mathcal{T}$ then, because $x \in \Lambda_K$, either $\sigma_\alpha(i) = s_i+\mathcal{T}$ or $\sigma_\beta(i) = s_i+\mathcal{T}$. Without loss of generality, suppose $\sigma_\alpha(i) = s_i+\mathcal{T}$. The coupling of $X$ with $\mathcal{Z}^\alpha$ holds for the time interval $[s_i,s_i+\mathcal{T}]$ so by Proposition \ref{AkProp}, with probability tending to 1 there will be at least $\mathcal{W}$ individuals with fitness in $[X_{s_i}^{\alpha(i)},\infty)$ at time $s_i+\mathcal{T}$. By Lemma \ref{brLem} at least two of those individuals will increase by at least 1 by time $s_{i+1}$. Since the probability of either of those individuals decreasing in fitness over the time interval $[s_i+\mathcal{T},s_{i+1}]$ is $(1-e^{-d\mathcal{T}})^2$ which tends to 0 as $N$ tends to infinity, with probability $p_2$ tending to 1 we have $X_{s_{i+1}}^{\beta(i+1)} > X_{s_i}^{\alpha(i)} \geq X_{s_i}^{\beta(i)}$.

Let $p = p_1 \wedge p_2$. Because there is only one point of $x$ over any time interval $[s_i,s_{i+1})$ we have $X_s^+ \geq X_{s_i}^{\beta(i)}$ for all $s \in [s_i,s_{i+1})$. This is where using the second highest fitness is important as we ensure that $X_s^+$ cannot decrease too far even when the maximum fitness class is lost. Since $x \in \Lambda_K$ we have $|x \cap [0,s_{M+1}]| = k \leq K$. The probability that $X_{s_{i+1}}^{\beta(i+1)} > X_{s_i}^{\beta(i)}$ on the $k$ intervals in which there is a point of $x$ is at least $p^K$. On an interval over which there is no point of $x$, we have $X_{s_{i+1}}^{\beta(i+1)} \geq X_{s_i}^{\beta(i)}$. Therefore, it is allowable for the fitness to fail to increase over some of the intervals on which there is no point of $x$. If $X_{s_{i+1}}^{\beta(i+1)} > X_{s_i}^{\beta(i)}$ for at least half of the values $i$ such that $|x \cap [s_i,s_{i+1})| = 0$ then $X^+$ must have increased by at least $(M-k)/2+k = (M+k)/2 > M/2$ for large enough values of $N$. Because the events $X_{s_{i+1}}^{\beta(i+1)} > X_{s_i}^{\beta(i)}$ are independent and the probability they occur is tending to 1, we can apply the Chernoff bound to get
$$P(\mathcal{E}|\mathcal{P} = x) > p^K(1- e^{-((M-K)/(2p)) \cdot (p-1/2)^2}).$$
As $N$ tends to infinity, $M$ tends to infinity and $p$ tends to 1, yielding the result.
\end{proof}

\begin{proof}[Proof of Theorem \ref{mainTheo}]
Let $\epsilon > 0$ and let $\delta > 0$ such that $(1-\delta)^2 > 1-\epsilon$. By Proposition \ref{PLk} there exists $K$ such that for $N$ large enough we have 
\begin{equation} \label{Peq}
P(\mathcal{P} \in \Lambda_K) > 1-\delta.
\end{equation}
We use the bound
$$P(\mathcal{E}) \geq P(\{\mathcal{P} \in \Lambda_K\} \cap \mathcal{E}) = E[1_{\{\mathcal{P} \in \Lambda_K\}}1_\mathcal{E}].$$
The set $\Lambda_K \in \mathcal{G}$ so $\{\mathcal{P} \in \Lambda_K\} \in \sigma(\mathcal{P})$. By definition of conditional expectation we have
$$E[1_{\{\mathcal{P} \in \Lambda_K\}} 1_\mathcal{E}] = E[1_{\{\mathcal{P} \in \Lambda_K\}} P(\mathcal{E}|\mathcal{P})].$$
By definition of expectation we can expand the expression as
$$E[1_{\{\mathcal{P} \in \Lambda_K\}} P(\mathcal{E}|\mathcal{P})] = \int_\Omega 1_{\Lambda_K}(\mathcal{P}(\omega)) P(\mathcal{E}|\mathcal{P} = \mathcal{P}(\omega)) dP(\omega).$$
For $G \in \mathcal{G}$ let
$$\mu_{\mathcal{P}}(G) = P(\mathcal{P} \in G).$$
Then we can push forward the integral to obtain
\begin{align*}
\int_\Omega 1_{\Lambda_K}(\mathcal{P}(\omega)) P(\mathcal{E}|\mathcal{P} = \mathcal{P}(\omega)) dP(\omega) & = \int_G 1_{\Lambda_K}(x) P(\mathcal{E} | \mathcal{P} = x) d\mu_{\mathcal{P}}(x) \\
& = \int_{\Lambda_K} P(\mathcal{E}|\mathcal{P} = x) d\mu_{\mathcal{P}}(x).
\end{align*}

By Proposition \ref{PEx} there exists $N_0$ such that if $N \geq N_0$ then $P(\mathcal{E}|\mathcal{P} = x) > 1-\delta$ for all $x \in \Lambda_K$. Hence
\begin{align*}
\int_{\Lambda_K} P(\mathcal{E}|\mathcal{P} = x) d\mu_{\mathcal{P}}(x) & \geq \int_{\Lambda_K} (1-\delta) d\mu_{\mathcal{P}}(x) \\
& = (1-\delta)\mu_{\mathcal{P}}(\Lambda_K) \\
& = (1-\delta) P(\mathcal{P} \in \Lambda_K) \\
& \geq (1-\delta)^2 \mbox{ by Equation \ref{Peq}}\\
& > 1-\epsilon.
\end{align*}

Therefore, for any $\epsilon > 0$ there exists $N_0$ such that if $N \geq N_0$ then $P(\mathcal{E}) \geq 1-\epsilon$ which completes the proof.
\end{proof}

\section{Acknowledgements}

I would like to thank my graduate advisor, Jason Schweinsberg, for introducing me to this model while I was pursuing my doctoral degree at U.C. San Diego. My ability to solve this problem undoubtedly stems from the many conversations we had about this model during the time I was there.


\begin{thebibliography}{9}
\bibitem{AN} K.B. Athreya, P.E. Ney (1972). \emph{Branching Processes}. Springer, New York

\bibitem{BRW} E. Brunet, I. Rouzine, C. Wilke (2008). The stochastic edge in adaptive evolution.  \emph{Genetics} \textbf{179}, 603-620.

\bibitem{DF} M. Desai, D.S. Fisher, (2007). Beneficial mutation-selection balance and the effect of linkage on positive selection. \emph{Genetics}  \textbf{176}, 1759-1798.

\bibitem{ROAUp} M. Kelly (2012). Upper Bound on the Rate of Adaptation in an Asexual Population. \emph{Ann. Appl. Probab.}, To appear.

\bibitem{RBC} I. Rouzine, E. Brunet, J. Coffin (2003). The solitary wave of asexual evolution. Proc. Natl. Acad. Sci. USA \textbf{100}, 587-592.

\bibitem{RBW} I. Rouzine, E. Brunet, C. Wilke (2007). The traveling-wave approach to asexual evolution: Muller's ratchet and speed of adaptation. \emph{Theor. Popul. Biol.} \textbf{73}, 24-46.

\bibitem{YEC} F. Yu, A. Etheridge, C. Cuthbertson (2010). Asymptotic Behaviour of the Rate of Adaptation. \emph{Ann. Appl. Probab.} \textbf{20}, 978-1004.
\end{thebibliography}
\end{document}